\documentclass[11pt]{amsart}
\usepackage[margin=3cm]{geometry}
\usepackage{mathrsfs,amsmath,amsfonts,amssymb,amsthm}
\usepackage{enumitem}
\usepackage{graphicx}
\usepackage[dvipsnames,svgnames,table]{xcolor}
\usepackage[linktocpage=true,colorlinks=true,linkcolor=RubineRed!85!black,citecolor=ForestGreen,urlcolor=green,pdfborder={0 0 0}]{hyperref}
\usepackage{dsfont}

\theoremstyle{plain}
\newtheorem{theorem}{Theorem}[section]

\newtheorem{proposition}[theorem]{Proposition}
\newtheorem{lemma}[theorem]{Lemma}
\theoremstyle{definition}

\numberwithin{equation}{section}

\newcommand*{\dif}{\mathop{}\!\mathrm{d}}
\def \R {\mathbb{R}}
\def \Sp {\mathbb{S}}
\def \HH {\mathrm{H}}

\def \II {\mathrm{I}}
\def \LL {L^{ij}_k}
\def \bb {\mathbf{b}}
\def \BB {\mathbf{B}^{ij}_k}
\def \ee {\mathbf{e}}

\def \ff {\mathbf{f}}
\def \Rs {R_{\rm\;\!spherical}}
\def \Dp {D_{\rm parallel}}
\def \Dr {D_{\rm radial}}
\def \Ds {D_{\rm spherical}}

\title{Fisher information for the multi-species Landau system}
\author{Yuzhe Zhu}
\date{December 19, 2025}
\thanks{The author thanks Luis Silvestre for helpful discussions.}

\begin{document}
\begin{abstract}
We consider of the Fisher information for solutions of spatially homogeneous multi-species Landau system. We show that the mass-weighted Fisher information is monotone decreasing in time along the Landau system with a general class of interaction potentials. 
\end{abstract}
\maketitle
\hypersetup{bookmarksdepth=2}
\setcounter{tocdepth}{1}
\tableofcontents

\section{Introduction}
We consider the spatially homogeneous kinetic model for the evolution of a dilute plasma or gas composed of $S\in\mathbb{N}_+$ distinct particle species. For each species indexed by $i\in\{1,\ldots,S\}$ with mass $m_i>0$, we denote by $f_i=f_i(t,v)>0$ the velocity distribution of its particles at time $t\in\R_+$ with velocity $v\in\R^d$ ($d\ge2$). For brevity, we collect the component distributions into the vector
\begin{align*}
\ff=\ff(t,v):=(f_1,\ldots,f_S)(t,v). 
\end{align*}
Of concern is the evolution of the distribution vector $\ff$ governed by the multi-species Landau system 
\begin{align}\label{multi-L}
\partial_tf_i = \sum_{j=1}^S Q_{ji}(f_j,f_i) {\quad\,\rm in\ \ }\R_+\times\R^d, \quad i\in\{1,\ldots,S\}. 
\end{align}
For $i,j\in\{1,\ldots,S\}$, the bilinear collision operator $Q_{ji}$ encodes binary interactions between species $j$ and $i$, and is defined by
\begin{align*}
Q_{ji}(f_j,f_i)&=\frac{c_{ji}}{m_i}\;\nabla_v \cdot \int_{\R^d} \alpha_{ji}(|v-v_*|)\;\!A(v-v_*)\left(\frac{1}{m_i}f_j^*\;\!\nabla_vf_i-\frac{1}{m_j}f_i\;\!\nabla_{v_*}f_j^*\right) \dif v_*,
\end{align*}
where we abbreviated $f_j^*=f_j(t,v_*)$. The coefficients satisfy
\begin{align*}
c_{ji}=c_{ij}\ge0. 
\end{align*}
The potential functions $\alpha_{ji}:\R_+\to\R_+$ are defined by
\begin{align*}
\alpha_{ji}(r):=r^{\gamma_{ji}}{\quad\rm for\ \ }\gamma_{ji}=\gamma_{ij}\in(-d-2,1],
\end{align*}
whose range makes the integral giving $Q_{ji}$ well-defined. The $d\times d$ matrix
\begin{align*}
A(v-v_*):=|v-v_*|^2I_d-(v-v_*)\otimes(v-v_*). 
\end{align*}

With the aim of generalizing the I-theorem from the single-species case ($S=1$) treated in \cite{Guillen-Silvestre} to mixtures, we establish that the Fisher information of the $S$-species system with the species masses $\{m_i\}_{i=1}^S$, defined by 
\begin{align*}
\II(\ff):=\sum_{i=1}^S\frac{1}{m_i}\int_{\R^d}\frac{|\nabla_vf_i|^2}{f_i}\dif v,
\end{align*}
is monotone decreasing along the solution $\ff$ to \eqref{multi-L} for a class of potential functions. 

\begin{theorem}\label{I}
Let the family of constants $\{\gamma_{ij}\}_{i,j=1}^S$ satisfy
\begin{align*}
|\gamma_{ij}|\le 
\begin{cases}
\,2\sqrt{d-1} &{\rm for\ \ } i\neq j,\\[4pt]
\,2\sqrt{d+3-\dfrac{1}{d-1}} &{\rm for\ \ } i=j. 
\end{cases}
\end{align*}
Then the Fisher information $\II(\ff)$ of any solution $\ff$ to \eqref{multi-L} is non-increasing in time. 
\end{theorem}

\subsection{Multi-species Landau system}
In physics, the symmetric coupling coefficients $c_{ji}$ appearing in the operators $Q_{ji}$ collect microscopic interaction prefactors determined by species properties such as mass and charge. The tensor $A(v-v_*)$ is a scaled orthogonal projection onto the hyperplane perpendicular to the relative velocity $v-v_*$, and captures that colliding particles undergo predominantly small-angle deviations. The exponent $\gamma_{ij}$ in $\alpha_{ij}$ classifies the interaction potential: $\gamma_{ij}\in(0,1]$ corresponds to hard potentials, $\gamma_{ij}=0$ to Maxwellian molecules, and $\gamma_{ij}\in(-d-2,0)$ to soft potentials. Of particular physical significance is the Coulomb interaction, given by $\gamma_{ij}=-3$ in dimension $d=3$. 

The prototypical configuration discussed in \cite{Landau} is the two-species electron-ion plasma system, and models arising from stellar systems can be found in \cite{Rosen,BT}. One may also refer to \cite[Chapter~7]{HCB} for the general kinetic theory of mixtures and to \cite[Chapter~4]{Lifshitz} for a comprehensive exposition of plasma collision processes in physical kinetics. 

Despite its mixture nature, the conservation and entropy structural features of \eqref{multi-L} are comparable to those in the single-species case ($S=1$). We briefly outline them below. 

\subsubsection{Conservation structure and equilibria}
Owing to the projection property of the tensor $A$, the divergence form of the collision operator $Q_{ji}$ and symmetric structure within \eqref{multi-L}, the system conserves, for each species $i\in\{1,\ldots,S\}$, its number density $n_i$, as well as the total momentum $P$ and energy $E$, 
\begin{align}\label{conserve1}
n_i:=\int_{\R^d} f_i\dif v,\quad
P:=\sum_{i=1}^Sm_i\int_{\R^d}vf_i\dif v,\quad
E:=\sum_{i=1}^S\frac{m_i}{2}\int_{\R^d}|v|^2f_i\dif v. 
\end{align}
From these quantities, one may further express the conserved macroscopic density $\rho$, bulk velocity $u$ and temperature $\theta$ of the mixture as 
\begin{align}\label{conserve2}
\rho:=\sum_{i=1}^Sm_in_i,\quad
u:=\frac{P}{\rho},\quad 
\theta:=\frac{\sum_{i=1}^Sm_i\int_{\R^d}|v-u|^2f_i\dif v}{d\sum_{i=1}^Sn_i}=\frac{2E-\rho|u|^2}{d\sum_{i=1}^Sn_i}. 
\end{align}
An equilibrium state of \eqref{multi-L} is the multi-species Maxwellian 
\begin{align}\label{conserved-eq}
\boldsymbol{\mu}(v):=(\mu_1,\ldots,\mu_S)(v),\quad
\mu_i(v):= n_i\left(\frac{m_i}{2\pi\theta}\right)^{d/2} \exp\left(-\frac{m_i\;\!|v-u|^2}{2\theta}\right), 
\end{align}
for prescribed conserved parameters $(n_i,u,\theta)$. One readily verifies that $Q_{ji}(\mu_j,\mu_i)=0$ for all $i,j\in\{1,\ldots,S\}$. 

\subsubsection{Entropy structure}
According to the additivity of entropy across species labels, the Boltzmann entropy of a multi-species state $\ff=(f_1,\ldots,f_S)$ is defined as 
\begin{align*}
\HH(\ff):=\sum_{i=1}^S\int_{\R^d} f_i\;\!\log f_i\dif v. 
\end{align*}
The evolution of this functional along solutions of \eqref{multi-L} is identical to that of the relative entropy with respect to the equilibrium $\boldsymbol{\mu}=(\mu_1,\ldots,\mu_S)$, 
\begin{align*}
\sum_{i=1}^S\HH(f_i\;\!|\;\!\mu_i),\quad \HH(f_i\;\!|\;\!\mu_i):=\int_{\R^d}f_i\;\!\log\frac{f_i}{\mu_i}\dif v,  
\end{align*} 
since they differ only by an additive constant fixed by the conserved quantities. 

Given the logarithmic form of the entropy, we abbreviate
\begin{align*}
\psi_i:=\log f_i,\quad \psi_i^*:=\log f_i^*,\quad i\in\{1,\ldots,S\}. 
\end{align*}
The collision operator $Q_{ji}$ can be then reformulated in its entropy form as  
\begin{align}\label{multi-LE} 
Q_{ji}(f_j,f_i)=c_{ji}\;\nabla_{m_iv}\cdot\int_{\R^d} \alpha_{ji}\;\!A\,f_if_j^*\left(\nabla_{m_iv}\psi_i-\nabla_{m_jv_*}\psi_j^*\right) \dif v_*. 
\end{align}
Here the notations $\nabla_{m_iv}=\frac{1}{m_i}\nabla_v$ and $\nabla_{m_jv_*}=\frac{1}{m_j}\nabla_{v_*}$ are inherited from the change of variables and represent momentum gradients.

\subsubsection{H-theorem via symmetrization}
Boltzmann’s H-theorem, formulated in the context of rarefied gas dynamics, asserts that the entropy functional $\HH$ decreases monotonically in time. In the present setting, the entropy dissipation identity follows from the standard symmetrization argument, originating in Boltzmann's computation \cite{Boltzmann}. More precisely, by applying \eqref{multi-L} and integration by parts, 
\begin{align*}
\frac{\dif}{\dif t}\,\HH(\ff)
&=\sum_{i,j=1}^S \int_{\R^d} \psi_i\, Q_{ji}(f_j,f_i)\dif v\\
&=-\sum_{i,j=1}^S c_{ji}  \int_{\R^{2d}} \alpha_{ji}\;\!A\,f_if_j^*\,\nabla_{m_iv} \psi_i\cdot\big(\nabla_{m_iv}\psi_i-\nabla_{m_jv_*}\psi_j^*\big) \dif v\dif v_*. 
\end{align*}
The symmetries $c_{ji}=c_{ij}$, $\alpha_{ji}=\alpha_{ij}$ and  $A(v-v_*)=A(v_*-v)$ ensure the above expression remains unchanged when swapping the variables $v\leftrightarrow v_*$ and the indices $i\leftrightarrow j$. Averaging the original and swapped forms then gives 
\begin{align}\label{multi-HH}
\frac{\dif}{\dif t}\,\HH(\ff)
=-\frac{1}{2}\sum_{i,j=1}^S c_{ij} \int_{\R^{2d}} \alpha_{ij}\;\!A\,f_if_j^* :\big(\nabla_{m_iv}\psi_i-\nabla_{m_jv_*}\psi_j^*\big)^{\otimes2} \dif v\dif v_*\le0. 
\end{align}
The entropy is monotone decreasing along solutions of \eqref{multi-L} and its production vanishes precisely at the equilibrium. 

\subsection{Multi-species Fisher information}
The Fisher information of concern for the mixture $\ff=(f_i,\ldots,f_S)$ is in the mass-weighted form, 
\begin{align*}
\II(\ff)=\sum_{i=1}^S\frac{1}{m_i}\II_{\;\!\rm s}(f_i), 
\end{align*}
where the Fisher information for a single species is defined as 
\begin{align*}
\II_{\;\!\rm s}(f_i):=\int_{\R^d}\frac{|\nabla_vf_i|^2}{f_i}\dif v
=\int_{\R^d}f_i\left|\nabla_v\log f_i\right|^2\dif v.
\end{align*}

\subsubsection{Linear Fokker-Planck variant}
The expression of the mass-weighted Fisher information can be heuristically understood by differentiating the entropy functional along a simplified linear diffusion-drift dynamics that admits the same equilibrium as the nonlinear kinetic system~\eqref{multi-L}.

The form of the multi-species equilibrium $\boldsymbol{\mu}=(\mu_1,\ldots,\mu_S)$ given in \eqref{conserved-eq} fixes the diffusion-drift balance in velocity space. A natural linear variant, sharing the same equilibrium state, is generated by the Fokker-Planck operator
\begin{align*}
\mathcal{L}_if_i:=c_i\,\nabla_v\cdot\left(\mu_i\,\nabla_v\frac{f_i}{\mu_i}\right)
=c_i\,\nabla_v\cdot\left(f_i\,\nabla_v\log\frac{f_i}{\mu_i}\right), 
\end{align*}
where $c_i>0$ is the diffusion coefficient. Assuming that the linear system
\begin{align*}
\partial_tf_i=\mathcal{L}_if_i, \quad i\in\{1,\ldots,S\}, 
\end{align*}  
conserves the quantities $n_i,P,E$ defined in \eqref{conserve1} and satisfies \eqref{conserve2}, one is able to determine that, modulo a common multiplicative constant, 
\begin{align*}
c_i=\frac{1}{m_i}.  
\end{align*} 
The entropy dissipation identity along this linear flow of each species reads
\begin{align*}
\frac{\dif}{\dif t}\,\HH(f_i\;\!|\;\!\mu_i)
=\int_{\R^d}\log\frac{f_i}{\mu_i}\,\mathcal{L}_if_i\dif v
=-\frac{1}{m_i}\int_{\R^d}f_i\,\bigg|\nabla_v\log\frac{f_i}{\mu_i}\bigg|^2\dif v. 
\end{align*} 
In terms of $\II_{\;\!\rm s}$, it can be reformulated as
\begin{align*}
\frac{\dif}{\dif t}\,\HH(f_i\;\!|\;\!\mu_i)
=-\frac{1}{m_i}\II_{\;\!\rm s}(f_i)-\frac{2d\;\!n_i}{\theta}+\frac{m_i}{\theta^2}\int_{\R^d}|v-u|^2f_i\dif v.  
\end{align*} 
Summing over $i=1,\ldots,S$ gives 
\begin{align*}
\frac{\dif}{\dif t}\,\HH(\ff)=\frac{\dif}{\dif t}\sum_{i=1}^S\HH(f_i\;\!|\;\!\mu_i)
=-\II(\ff)- \frac{d}{\theta}\sum_{i=1}^S n_i. 
\end{align*} 
This viewpoint clarifies the role of mass weighting in the sum over species, showing that it aligns the entropy dissipation along the linear flow with the multi-species Fisher information, up to an additive constant fixed by the conserved quantities. 

\subsubsection{I-theorem via symmetrization}
In the context of single-species dynamics, the seminal contributions of \cite{Guillen-Silvestre} and \cite{ISV} established the validity of the I-theorem. It identifies the Fisher information as an additional decreasing Lyapunov functional for both the Landau and Boltzmann collision operators, alongside the entropy in the classical H-theorem. 

Motivated by these developments, we broaden the scope to multi-species dynamics. We first propose extending the symmetrization argument, previously employed to derive the entropy dissipation identity \eqref{multi-HH}, to the analysis of the Fisher information dissipation for \eqref{multi-L}. The symmetrization-based approach not only recovers and generalizes the computation performed via the lifting technique developed in \cite{Guillen-Silvestre}, but also offers greater flexibility when applied to diverse classes of kinetic models with highly nonlinear collision statistics, such as those of Fermi-Dirac \cite[Chapter~17]{CC} and Lenard-Balescu \cite[Chapter~4]{Lifshitz} type. 

In Section~\ref{section-Fisher}, we carry out the computation of the Fisher information dissipation. By generalizing the vector field formulation used in \cite{Guillen-Silvestre}, we describe the associated differential structure and the underlying invariant properties, and show how these features persist across multi-species interactions. 

\subsection{Parametrization and geometric reduction}
One of the key observations in \cite{Guillen-Silvestre} for the single-species case is that, upon adopting an appropriate parametrization, the monotonicity of the Fisher information can be ensured by a geometric inequality that admits an intrinsic interpretation within the Bakry-\'Emery formalism. 

\subsubsection{Microscopic parametrization}\label{mp}
The properties of the collision operator are intimately tied to the kinematics of binary elastic collisions. Consider two particles of species $i$ and $j$ with pre-collision velocities $(v,v_*)\in\R^{2d}$. At this microscopic level, the fundamental conservation laws of momentum and energy constrain the post-collision velocities $(v',v_*')\in\R^{2d}$ to the (multi-species) Boltzmann collision manifold
\begin{align*}
\mathcal{M}_{ij}(v,v_*):=\big\{\,(v',v_*')\in\R^{2d}:\ m_iv+m_jv_*&=m_iv'+m_jv_*',\\\ m_i|v|^2+m_j|v_*|^2&=m_i|v'|^2+m_j|v_*'|^2 \,\big\}. 
\end{align*}
This $(d-1)$-dimensional manifold is commonly parameterized by an angular variable $\sigma\in\Sp^{d-1}$. In terms of the center-of-mass velocity $z\in\R^d$ and the relative speed $r\in\R_+$ defined by
\begin{align*}
z:= \frac{m_iv+m_jv_*}{m_i+m_j}, \quad r:=|v-v_*|, 
\end{align*}
the post-collision velocities can be written as 
\begin{align}\label{collision}
\left\{\begin{aligned}
&\ v'= z +\frac{m_j}{m_i+m_j}\;\!r\;\!\sigma,\\
&\ v_*'= z -\frac{m_i}{m_i+m_j}\;\!r\;\!\sigma. 
\end{aligned}\right. 
\end{align}
This constitutes the standard $\sigma$-representation, which highlights that the collision acts as a rotation of the relative velocity in the center-of-mass frame, leaving $(z,r)$ invariant. 

Accordingly, the microscopic parametrization \eqref{collision} endows the doubled velocity space $\mathbb{R}^{2d}$ with a natural coordinate system given by $(z,r,\sigma)\in\R^d\times\R_+\times\Sp^{d-1}$. In Section~\ref{section-inequality}, passing to $(z,r,\sigma)$ makes explicit the reduction to a spherical geometric setting. 

\subsubsection{Geometric inequality on the sphere}
Upon the microscopic parametrization, the Fisher information dissipation can be reorganized so that the genuinely delicate part lives in the angular variable $\sigma\in\Sp^{d-1}$. It was provided in \cite{Guillen-Silvestre} a sufficient condition reducing the nonnegativity of the dissipation to a purely geometric inequality on the sphere, in a manner analogous to the classical Bakry-\'Emery $\Gamma_2$ criterion for diffusion operators (see \cite{BE,BGL}). The resulting monotonicity statement holds only for those interaction potentials $\alpha_{ji}$ whose parameters $\gamma_{ij}$ fall below the best constant in that inequality. 

A key input in the single-species setting is the antipodal symmetry of the collision geometry: the microscopic collision process does not distinguish $\sigma$ from $-\sigma$ within the same species. Equivalently, for the velocity distributions of species indexed by $i,j\in\{1,\ldots,S\}$, 
\begin{align}\label{fij}
f_i(v)f_j(v_*)=f_i(v_*)f_j(v) {\quad\rm if\ \,} i=j. 
\end{align}
In particular, this identity is automatic when $S=1$. Based on this symmetric condition, \cite{Guillen-Silvestre,Ji} refined the Bakry-\'Emery geometric criterion by passing to the real projective space $\Sp^{d-1}/\{\pm I\}$, which enlarges the admissible class of potentials and, in the single-species case, allows one to include the Coulomb interaction. 

For mixtures, however, interactions between distinct species introduce an intrinsic asymmetry: when $i\neq j$, the relation \eqref{fij} generally fails. This loss of symmetry obstructs any further improvement of the geometric criterion. In particular, the sharp constant available in that framework does not cover the Coulomb interaction for cross-species collisions, and the spherical geometry exploited in \cite{Guillen-Silvestre} no longer suffices to recover a definite sign for the Fisher information dissipation.

In Section~\ref{section-inequality}, we state the strongest conclusions one can reasonably expect from the spherical geometric approach, as dictated by the optimal constant in the underlying inequality. Numerical tests suggest that this threshold is essentially optimal, in the sense that the dissipation may fail to have a definite sign beyond it. 

\section{Fisher information along the Landau system}\label{section-Fisher}
This section is devoted to the derivation of the Fisher information dissipation formula in Proposition~\ref{diss-multiL}, stated further below. 

\subsection{Differentiation along the Landau system}
Through symmetrization, we first derive a formula for the time derivative of the Fisher information along the Landau system.
\begin{lemma}\label{I11}
Let $\ff=(f_1,\ldots,f_S)$ be a solution to \eqref{multi-L}, and let $\psi_i:=\log f_i$ and $\psi_i^*:=\log f_i^*$. We have 
\begin{align*}
\frac{\dif}{\dif t}\,\II(\ff)
=\sum_{i,j=1}^S\frac{c_{ij}}{2}\int_{\R^{2d}} \alpha_{ij}\;\!A\;\!f_if_j^*\left(\nabla_{m_iv}\psi_i-\nabla_{m_jv_*}\psi_j^*\right)\cdot(\nabla_{m_iv}-\nabla_{m_jv_*})\left(\frac{\Xi_i}{m_i}+\frac{\Xi_j^*}{m_j} \right), 
\end{align*}
where $\alpha_{ij}=\alpha_{ij}(|v-v_*|)$, $A=A(v-v_*)$, and 
\begin{align}\label{phiv}
\begin{aligned}
\Xi_i&:=2\,\Delta_v\psi_i + |\nabla_v\psi_i|^2, \\
\Xi_i^*&:=2\,\Delta_{v_*}\psi_i^* + |\nabla_{v_*}\psi_i^*|^2.  
\end{aligned}
\end{align}
\end{lemma}

\begin{proof}
Along the Fisher information for each species in \eqref{multi-L}, a direct computation with integration by parts yields that 
\begin{align*}
\frac{\dif}{\dif t}\,\II_{\;\!\rm s}(f_i)
&=\sum_{j=1}^S \int_{\R^d} |\nabla_v\psi_i|^2\;\!Q_{ji}(f_j,f_i) + 2f_i\,\nabla_v\psi_i\cdot\nabla_v\left(f_i^{-1}Q_{ji}(f_j,f_i)\right)\\
&=\sum_{j=1}^S \int_{\R^d} |\nabla_v\psi_i|^2\;\!Q_{ji}(f_j,f_i) - 2\left(\Delta_v\psi_i +f_i^{-1}\nabla_v f_i\cdot\nabla_v\psi_i\right) Q_{ji}(f_j,f_i). 
\end{align*}
By rearranging terms and using \eqref{multi-LE}, this becomes
\begin{align*}
\frac{\dif}{\dif t}\,\II_{\;\!\rm s}(f_i)
=-\sum_{j=1}^S \int_{\R^d} Q_{ji}(f_j,f_i)\;\Xi_i\dif v. 
\end{align*}
Invoking \eqref{phiv} and performing an integration by parts, we obtain 
\begin{align*}
\frac{\dif}{\dif t}\,\II_{\;\!\rm s}(f_i)
&=\sum_{j=1}^S c_{ji} \int_{\R^{2d}} \alpha_{ji}\;\!A \,f_if_j^*\left(\nabla_{m_iv}\psi_i-\nabla_{m_jv_*}\psi_i^*\right)\cdot\nabla_{m_iv}\Xi_i \dif v\dif v_*. 
\end{align*}
Upon summing over $i$ with the weight $1/m_i$, we note that the expression remains unchanged when the variables $v$ and $v_*$, as well as the indices $i$ and $j$, are swapped, thereby yielding 
\begin{align*}
\frac{\dif}{\dif t}\,\II(\ff)
&=\sum_{i,j=1}^S c_{ij} \int_{\R^{2d}} \alpha_{ij}\;\!A \,f_if_j^*\left(\nabla_{m_iv}\psi_i-\nabla_{m_jv_*}\psi_i^*\right)\cdot\nabla_{m_iv}\left(\frac{\Xi_i}{m_i}\right) \dif v\dif v_*\\
&=\sum_{i,j=1}^S c_{ij} \int_{\R^{2d}} \alpha_{ij}\;\!A\,f_if_j^*\left(\nabla_{m_jv_*}\psi^*-\nabla_{m_iv}\psi\right)\cdot\nabla_{m_jv_*} \left(\frac{\Xi_j^*}{m_j}\right) \dif v\dif v_*, 
\end{align*}
where the facts that $c_{ji}=c_{ij}$, $\alpha_{ji}=\alpha_{ij}$ and  $A(v-v_*)=A(v_*-v)$ were used. By averaging, we arrive at the desired symmetrized form. 
\end{proof}

\subsection{Tensorized functions}
Given the symmetrized form in Lemma~\ref{I11}, it is convenient to tensorize the pair of one-particle densities $(f_i,f_j^*)$ into a single function on the doubled velocity space $(v,v_*)\in\R^d\times\R^d$. More precisely, for each ordered pair of species $(i,j)$, we define 
\begin{align*}
F_{ij}(v,v_*)&:=f_if_j^*=f_i(v)f_j(v_*). 
\end{align*}
This tensorization further shows, as observed in \cite{Guillen-Silvestre}, that the nonlocal structure of the collision operator can be realized as a second-order differential operator. 
Accordingly, we set 
\begin{align*}
\Psi_{ij}(v,v_*)&:=\log(F_{ij})=\log (f_if_j^*)=\psi_i+\psi_j^*. 
\end{align*}
Recalling \eqref{phiv}, we have
\begin{align*}
\frac{\Xi_i}{m_i} + \frac{\Xi_j^*}{m_j} =2\left(\frac{\Delta_v}{m_i}+\frac{\Delta_{v_*}}{m_j}\right)\Psi_{ij} + \frac{|\nabla_v\Psi_{ij}|^2}{m_i} + \frac{|\nabla_{v_*}\Psi_{ij}|^2}{m_j}. 
\end{align*}
With the tensorized notation, the formula in Lemma~\ref{I11} can be thus recast as 
\begin{align}\label{I111}
\frac{\dif}{\dif t}\,\II(\ff)
=\sum_{i,j=1}^S c_{ij} \int_{\R^{2d}}  \alpha_{ij}\;\!F_{ij}\;\!A\left(\nabla_{m_iv}-\nabla_{m_jv_*}\right)\Psi_{ij}\cdot(\nabla_{m_iv}-\nabla_{m_jv_*}) \left(\frac{\Delta_v}{m_i}+\frac{\Delta_{v_*}}{m_j}\right)\Psi_{ij} \nonumber\\
+\frac{\alpha_{ij}}{2}\,F_{ij}\;\!A\left(\nabla_{m_iv}-\nabla_{m_jv_*}\right)\Psi_{ij}\cdot(\nabla_{m_iv}-\nabla_{m_jv_*}) \left(\frac{|\nabla_v\Psi_{ij}|^2}{m_i} + \frac{|\nabla_{v_*}\Psi_{ij}|^2}{m_j}\right) &. 
\end{align}
As it is clear that the integral is taken over $(v,v_*)\in\R^{2d}$, we henceforth omit the integration element $\dif v\dif v_*$.

\subsection{Vector field formulation}\label{vf}
We generalize the vector-field formulation of the Landau collision operator, a tractable analytic framework developed in \cite{Guillen-Silvestre}. 

Consider the special orthogonal Lie algebra $\mathfrak{so}(d)$, consisting of all $d\times d$ skew-symmetric matrices. A standard basis of $\mathfrak{so}(d)$ is given by $\{E_{pq}-E_{qp}\}_{1\le p<q\le d}$, where $E_{pq}$ is the $d\times d$ matrix with a $1$ in the $(p,q)$-entry and zeros elsewhere. The dimension of $\mathfrak{so}(d)$ is 
\begin{align*}
M:=\left|\{E_{pq}-E_{qp}:1\le p<q\le d\}\right|=\frac{d\left(d-1\right)}{2}. 
\end{align*}
Fix an enumeration $\{(p_k,q_k)\}_{k=1}^M=\{(p,q):1\le p<q\le d\}$ of all such index pairs. For $v,v_*\in\R^d$, define the vector fields
\begin{align}\label{bbk}
\begin{aligned}
\bb_k:\!&=(E_{p_kq_k}-E_{q_kp_k})(v-v_*)\\
&=(v-v_*)_{q_k}\;\!\ee_{p_k}-(v-v_*)_{p_k}\;\!\ee_{q_k}, 
\end{aligned}
\end{align}
where $\{\ee_p\}_{p=1}^d$ denotes the standard orthonormal basis of $\R^d$. 

As $|z|^{-2}A(z)$ is the orthogonal projection onto $(\R z)^\perp$, it admits the tensor representation 
\begin{align}\label{AA} 
A(v-v_*)=\sum_{k=1}^M\bb_k\otimes\bb_k. 
\end{align}
For $k=1,2,\dots,M$, by setting 
\begin{align*}
\BB:=\begin{pmatrix}\bb_k/m_i\\-\bb_k/m_j\end{pmatrix}, 
\end{align*}
we consider the directional derivative  
\begin{align}\label{LLk}
\LL:=\BB\cdot\nabla
=\frac{1}{m_i}\bb_k\cdot\nabla_v -\frac{1}{m_j}\bb_k\cdot\nabla_{v_*}. 
\end{align}
It generates the rotations in the relative velocity variable $v-v_*$ in the $(p_k,q_k)$-plane about the center $m_iv+m_jv_*$. 
Let us enumerate some of its fundamental properties. 

\begin{enumerate}[label=(\roman*), leftmargin=*, labelsep=0.5em]
\item Since $\BB$ is divergence-free, $\LL$ is skew-adjoint. 

\item\label{property-skew} The matrix $D\BB$ induced by the commutator $[\nabla,\LL]$ is skew-symmetric. 

\item Since $\bb_k$ is perpendicular to $v-v_*$, $\LL$ annihilates all the functions depending only on $|v-v_*|$. 
The center $m_iv+m_jv_*$ is also invariant under the action of $\LL$. 
More concretely, for any differentiable function $g=g(|v-v_*|,m_iv+m_jv_*)$, we have
\begin{align*}
\LL g=0.
\end{align*} 
Here $m_iv+m_jv_*$ can be interpreted as the total momentum of two interacting particles of species $i$ and $j$. Recalling the center-of-mass velocity $z:=\frac{m_iv+m_jv_*}{m_i+m_j}$, the total kinetic energy is another invariance of $\LL$, as it can be expressed as 
\begin{align*}
m_i|v|^2+m_j|v_*|^2
&=m_i\left|z+\frac{m_j(v-v_*)}{m_i+m_j}\right|^2 + m_j\left|z-\frac{m_i(v-v_*)}{m_i+m_j}\right|^2\\
&=\frac{|m_iv+m_jv_*|^2}{m_i+m_j} + \frac{m_i\;\!m_j}{m_i+m_j}\,|v-v_*|^2. 
\end{align*} 

\item Now that the flow of $\BB$ rotates the relative velocity $v-v_*$ while preserving both the momentum $m_iv+m_jv_*$ and the energy $m_i|v|^2+m_j|v_*|^2$, the Boltzmann collision manifold $\mathcal{M}_{ij}(v,v_*)$ (see the definition in \S~\ref{mp}) coincides with the collection of points reachable by flowing along $\BB$, namely 
\begin{align*}
\qquad\ \ 
\left\{\begin{aligned}
\begin{pmatrix}v(t)\\v_*(t)\end{pmatrix}\!:
\begin{pmatrix}\dot{v}(t)\\ \dot{v}_*(t)\end{pmatrix}\!= \BB\left(v(t)-v_*(t)\right),\,\begin{pmatrix}v(0)\\v_*(0)\end{pmatrix}=\begin{pmatrix}v\\v_*\end{pmatrix},\,t\in\R,\,1\le k\le M 
\end{aligned}\right\}. 
\end{align*}

\item\label{property-comm} The operator $\LL$ commutes with the mass-weighted Laplacian 
\begin{align*}
\Delta_{m_i,m_j}:=\frac{\Delta_v}{m_i}+\frac{\Delta_{v_*}}{m_j},  
\end{align*}
that is, 
\begin{align*}
\big[\,\Delta_{m_i,m_j},\,\LL\,\big]=0. 
\end{align*}
\end{enumerate}

\subsection{Fisher information dissipation}
For convenience, we introduce the notation
\begin{align}\label{mm}
\begin{aligned}
\nabla_{\sqrt{m}_i,\sqrt{m}_j}&:=(\nabla_{\sqrt{m}_iv},\nabla_{\sqrt{m}_jv_*}),\\
\nabla_{m_i,m_j}&:=(\nabla_{m_iv},\nabla_{m_jv_*}). 
\end{aligned}
\end{align}
Together with the vector fields defined in Subsection~\ref{vf}, we aim to establish the following formula for the Fisher information dissipation along the Landau system. 

\begin{proposition}\label{diss-multiL}
Let $\ff$ be a solution to \eqref{multi-L}. Then, 
\begin{align*}
\frac{\dif}{\dif t}\,\II(\ff) =\sum_{i,j=1}^S\sum_{k=1}^M c_{ij}\! \int_{\R^{2d}}\! - F_{ij}\,\big|\nabla_{\sqrt{m}_i,\sqrt{m}_j}(\sqrt{\alpha}_{ij}\;\!\LL\Psi_{ij})\big|^2  + \frac{m_i+m_j}{m_i\;\!m_j}\frac{(\alpha_{ij}')^2}{4\alpha_{ij}} F_{ij}\,\big|\LL\Psi_{ij}\big|^2.
\end{align*}
\end{proposition}

\begin{proof}
Invoking \eqref{AA} and \eqref{LLk}, we recast the expression~\eqref{I111} for the derivative of the Fisher information, which is a reformulation of Lemma~\ref{I11}, as 
\begin{align}\label{I111-k}
\begin{aligned}
\frac{\dif}{\dif t}\,\II(\ff)
=\sum_{i,j=1}^S\sum_{k=1}^M c_{ij}\int_{\R^{2d}} &\  \alpha_{ij}\;\!F_{ij}\;\!\LL\Psi_{ij}\,\LL\,\Delta_{m_i,m_j}\,\Psi_{ij} \\
&+\frac{\alpha_{ij}}{2}\,F_{ij}\;\!\LL\Psi_{ij}\,\LL\left(\frac{|\nabla_v\Psi_{ij}|^2}{m_i} + \frac{|\nabla_{v_*}\Psi_{ij}|^2}{m_j}\right). 
\end{aligned}
\end{align}
Our focus now turns to the first term on the right-hand side above, that is, 
\begin{align*}
T_{ij}^k &:= \int_{\R^{2d}} \alpha_{ij}\;\!F_{ij}\;\!\LL\Psi_{ij}\,\LL\,\Delta_{m_i,m_j}\,\Psi_{ij}. 
\end{align*}
In view of the commutation relation stated in Property~\ref{property-comm} above, it can be rewritten as 
\begin{align*}
T_{ij}^k =\int_{\R^{2d}}\sqrt{\alpha}_{ij}\;\!F_{ij}\,\LL\Psi_{ij}\,\Delta_{m_i,m_j} (\sqrt{\alpha}_{ij}\;\!\LL\Psi_{ij}) + \sqrt{\alpha}_{ij}\;\!F_{ij}\,\LL\Psi_{ij}\,\big[\sqrt{\alpha}_{ij},\Delta_{m_i,m_j}\big](\LL\Psi_{ij}).
\end{align*}
Upon integrating by parts term by term, and using the abbreviations in \eqref{mm}, we obtain 
\begin{align*}
T_{ij}^k=-\!\int_{\R^{2d}}& F_{ij}\, \big|\nabla_{\sqrt{m}_i,\sqrt{m}_j}(\sqrt{\alpha}_{ij}\;\!\LL\Psi_{ij})\big|^2
+ \sqrt{\alpha}_{ij}\;\!\LL\Psi_{ij}\,\nabla_{m_i,m_j}F_{ij}\cdot\nabla (\sqrt{\alpha}_{ij}\;\!\LL\Psi_{ij}) \\ 
 & + \nabla_{m_i,m_j}(\alpha_{ij}\;\!F_{ij}\,\LL\Psi_{ij})\cdot\nabla\LL\Psi_{ij}
-\nabla_{m_i,m_j}(\sqrt{\alpha}_{ij}\;\!F_{ij}\,\LL\Psi_{ij})\cdot\nabla(\sqrt{\alpha}_{ij}\;\!\LL\Psi_{ij}). 
\end{align*}
The last three terms on the right-hand side above simplify to
\begin{align*}
-\left(\frac{|\nabla_v\sqrt{\alpha}_{ij}|^2}{m_i}+\frac{|\nabla_{v_*}\sqrt{\alpha}_{ij}|^2}{m_j}\right)F_{ij}\,\big|\LL\Psi_{ij}\big|^2 + \alpha_{ij}\;\!\LL\Psi_{ij}\,\nabla_{m_i,m_j}F_{ij}\cdot\nabla\LL\Psi_{ij}, 
\end{align*}
by applying the elementary identity due to the product rule, 
\begin{align*}
s_0s_2 Ds_1\cdot D(s_0s_2) +  D(s_0^2s_1s_2)\cdot Ds_2 - D(s_0s_1s_2)\cdot D(s_0s_2)\\ = - s_1 s_2^2|Ds_0|^2 + s_0^2s_2(Ds_1\cdot Ds_2) &, 
\end{align*}
with the substitution of the scalar functions $(s_0,s_1,s_2)=(\sqrt{\alpha}_{ij},F_{ij},\LL\Psi_{ij})$, and with the differential operator $D=\nabla_v$ or $\nabla_{v_*}$.  
By noticing $|\nabla_v\sqrt{\alpha}_{ij}|^2=|\nabla_{v_*}\sqrt{\alpha}_{ij}|^2=(\alpha_{ij}')^2/(4\alpha_{ij})$, we arrive at 
\begin{align*}
T_{ij}^k = \int_{\R^{2d}}& -F_{ij}\, \big|\nabla_{\sqrt{m}_i,\sqrt{m}_j}(\sqrt{\alpha}_{ij}\;\!\LL\Psi_{ij})\big|^2 \\ 
&+ \frac{m_i+m_j}{m_i\;\!m_j}\frac{(\alpha_{ij}')^2}{4\alpha_{ij}} F_{ij}\,\big|\LL\Psi_{ij}\big|^2 -\alpha_{ij}\;\!\LL\Psi_{ij}\,\nabla_{m_i,m_j}F_{ij}\cdot\nabla\LL\Psi_{ij}. 
\end{align*}
We then deduce from this identity and \eqref{I111-k} that 
\begin{align*}
\frac{\dif}{\dif t}\,\II(\ff) =\sum_{i,j=1}^S\sum_{k=1}^M c_{ij} \int_{\R^{2d}} - F_{ij}\,\big|\nabla_{\sqrt{m}_i,\sqrt{m}_j}(\sqrt{\alpha}_{ij}\;\!\LL\Psi_{ij})\big|^2  +R_{ij}^k,
\end{align*}
where the remainder terms are collected as 
\begin{align}\label{Rem-ijk}
R_{ij}^k:=&\, \frac{m_i+m_j}{m_i\;\!m_j}\frac{(\alpha_{ij}')^2}{4\alpha_{ij}} F_{ij}\,\big|\LL\Psi_{ij}\big|^2\nonumber\\ &-\alpha_{ij}\;\!\LL\Psi_{ij}\,\nabla_{m_i,m_j}F_{ij}\cdot\nabla\LL\Psi_{ij} +\frac{\alpha_{ij}}{2}\,F_{ij}\;\!\LL\Psi_{ij}\,\LL\left(\frac{|\nabla_v\Psi_{ij}|^2}{m_i} + \frac{|\nabla_{v_*}\Psi_{ij}|^2}{m_j}\right). 
\end{align} 
By the definition of $\Psi_{ij}$ and the skew-symmetry of $D\BB$ (see also  Property~\ref{property-skew}), we have 
\begin{align*}
\nabla_{m_i,m_j}F_{ij}\cdot\nabla\LL\Psi_{ij}
&=F_{ij}\left(\frac{\nabla_v\Psi_{ij}}{m_i},\frac{\nabla_{v_*}\Psi_{ij}}{m_j}\right)\cdot\nabla\LL\Psi_{ij}\\
&=F_{ij}\left(\frac{\nabla_v\Psi_{ij}}{m_i},\frac{\nabla_{v_*}\Psi_{ij}}{m_j}\right) \cdot\LL\nabla\Psi_{ij} \\
&=\frac{1}{2}\,\LL\left(\frac{|\nabla_v\Psi_{ij}|^2}{m_i} + \frac{|\nabla_{v_*}\Psi_{ij}|^2}{m_j}\right),  
\end{align*} 
which implies that the last two terms on the right-hand side of \eqref{Rem-ijk} cancel identically. We therefore conclude the formula as claimed. 
\end{proof}

\section{Spherical geometric criterion}\label{section-inequality}

\subsection{Change of variables}
Let us fix the pair of indices $(i,j)$ with $i,j=1,\ldots,S$. Recall the motivation underlying the microscopic parametrization described in \S~\ref{mp}. We consider the change of the variables from $(v,v_*)\in\R^{2d}$ to $(z,r,\sigma)\in\R^d\times\R_+\times\Sp^{d-1}$, specified by 
\begin{align}\label{change}
\left\{\begin{aligned}
&\; z=\frac{m_i\;\!v+m_j\;\!v_*}{m_i+m_j}\\
&\ r\;\!\sigma = v-v_*
\end{aligned}\right. 
\quad\Longleftrightarrow\quad
\left\{\begin{aligned}
&\ v=z+\frac{m_j}{m_i+m_j}\;\!r\;\!\sigma\\
&\ v_*=z-\frac{m_i}{m_i+m_j}\;\!r\;\!\sigma .
\end{aligned}\right. 
\end{align} 
In particular, $|v-v_*|=r$. Recalling \eqref{bbk}, we see that the vector $\bb_k$, with $1\le k\le M=d\left(d-1\right)/2$, takes the form
\begin{align*}
&\bb_k= r\;\!b_k, \\
&b_k:=(E_{p_kq_k}-E_{q_kp_k})\,\sigma=\sigma_{q_k}\;\!\ee_{p_k}-\sigma_{p_k}\;\!\ee_{q_k}.
\end{align*}
From the definition \eqref{LLk}, and the orthogonality $b_k\cdot\sigma=0$, it follows that 
\begin{align}\label{LLr}
\LL=\frac{m_i+m_j}{m_i\;\!m_j}\,r\;\!b_k\cdot\nabla_{r\sigma}
= \frac{m_i+m_j}{m_i\;\!m_j}\, b_k\cdot\nabla_\sigma. 
\end{align}

The change of variables further gives the integration element 
\begin{align*}
\dif v\dif v_*=r^{d-1}\dif z\dif r\dif\sigma, 
\end{align*}
and the differentiation relations 
\begin{align*}
\nabla_v&= \frac{m_i}{m_i+m_j}\;\!\nabla_z +\nabla_{r\sigma}= \frac{m_i}{m_i+m_j}\;\!\nabla_z + \sigma\;\!\partial_r+ \frac{1}{r}\,\nabla_\sigma, \\
\nabla_{v_*}&= \frac{m_j}{m_i+m_j}\;\!\nabla_z -\nabla_{r\sigma} = \frac{m_j}{m_i+m_j}\;\!\nabla_z-\sigma\;\!\partial_r-\frac{1}{r}\,\nabla_\sigma. 
\end{align*}
Here the gradient $\nabla_\sigma$ on $\Sp^{d-1}$ coincides with the orthogonal projection of the Euclidean gradient $\nabla_{\R^d}$ onto the tangent space $T_\sigma\Sp^{d-1}$. For any differentiable function $\Phi:\Sp^{d-1}\rightarrow\R$, we have $\nabla_\sigma\Phi\in T_\sigma\Sp^{d-1}$ and $\sigma\cdot\nabla_\sigma\Phi=0$. 

\subsection{Bakry-\'Emery formalism}
Following \cite{Guillen-Silvestre}, the formula of the Fisher information dissipation in Proposition~\ref{diss-multiL} fits naturally within the Bakry-\'Emery formalism (see for instance \cite{BE,BGL}), which we briefly recall for completeness. 

Since $\sum_{k=1}^M b_k\otimes b_k$ is the orthogonal projection onto $T_\sigma\Sp^{d-1}$, for any smooth function, we have  $\Phi:\Sp^{d-1}\rightarrow\R$,   
\begin{align*}
\nabla_\sigma\Phi=\sum_{k=1}^M b_k\,(b_k\cdot\nabla_\sigma\Phi).
\end{align*}
It follows that 
\begin{align*}
|\nabla_\sigma\Phi|^2 &= \sum_{k=1}^M \left|b_k\cdot\nabla_\sigma\Phi\right|^2,\\
\Delta_\sigma\Phi &= \sum_{k=1}^M b_k\cdot\nabla_\sigma(b_k\cdot\nabla_\sigma\Phi). 
\end{align*}
Moreover, by noticing that $\langle\nabla_\sigma^2\Phi(\cdot),\nabla_\sigma b_k(\cdot)\rangle$ is a skew-symmetric bilinear form on $T_\sigma\Sp^{d-1}$, and $\Sp^{d-1}$ with the standard round metric $g_{\rm\;\!round}$ has constant Ricci curvature $(d-2)\;\!g_{\rm\;\!round}$, we have
\begin{align*}
\sum_{k=1}^M |\nabla_\sigma(b_k\cdot\nabla_\sigma\Phi)|^2 
&= \sum_{k=1}^M \big|\nabla_\sigma^2\Phi[b_k,\cdot]\big|^2 + |(\nabla_\sigma b_k)\nabla_\sigma\Phi|^2\\
&=\big|\nabla_\sigma^2\Phi\big|_{\rm HS}^2 +\mathrm{Ric}\;\!(\nabla_\sigma\Phi,\nabla_\sigma\Phi)
=\big|\nabla_\sigma^2\Phi\big|_{\rm HS}^2 +(d-2)|\nabla_\sigma\Phi|^2,  
\end{align*}
where $|\cdot|_{\rm HS}$ denotes the Hilbert-Schmidt norm. 

The above identities can be interpreted in the language of Bakry-\'Emery's $\Gamma$-calculus by setting the carr\'e du champ operator $\Gamma$ and its iterated version $\Gamma_2$ on $\Sp^{d-1}$,
\begin{align*}
\Gamma(\Phi,\Psi)&:= \frac{1}{2}\left(\Delta_\sigma(\Phi\;\!\Psi) -\Psi\;\!\Delta_\sigma\Phi -\Phi\;\!\Delta_\sigma\Psi\right), \\
\Gamma_2(\Phi,\Psi)&:= \frac{1}{2}\left(\Delta_\sigma\Gamma(\Phi,\Psi) -\Gamma(\Phi,\Delta_\sigma\Psi) -\Gamma(\Delta_\sigma\Phi,\Psi)\right).
\end{align*}
By abbreviating $\Gamma(\Phi):=\Gamma(\Phi,\Phi)$ and $\Gamma_2(\Phi):=\Gamma_2(\Phi,\Phi)$, we have 
\begin{align}\label{BE1}
\begin{aligned}
\Gamma(\Phi)&=|\nabla_\sigma\Phi|^2=\sum_{k=1}^M \left|b_k\cdot\nabla_\sigma\Phi\right|^2,\\
\Gamma_2(\Phi)&=\frac{1}{2}\left(\Delta_\sigma\Gamma(\Phi)-2\,\Gamma(\Phi,\Delta_\sigma\Phi)\right), 
\end{aligned}
\end{align}
from which Bochner's formula implies 
\begin{align}\label{BE2}
\Gamma_2(\Phi)=\big|\nabla_\sigma^2\Phi\big|_{\rm HS}^2 +(d-2)|\nabla_\sigma\Phi|^2
=\sum_{k=1}^M |\nabla_\sigma(b_k\cdot\nabla_\sigma\Phi)|^2 .
\end{align}

\subsection{Geometric decomposition}
Combining the parametrization with the Bakry-\'Emery reformulation introduced above yields the following equivalent statement of Proposition~\ref{diss-multiL}. 

\begin{proposition}\label{Fisher-decom}
Let $\ff$ be a solution to \eqref{multi-L}. Then, 
\begin{align}
\frac{\dif}{\dif t}\,\II(\ff) = \sum_{i,j=1}^S\tilde{c}_{ij}\left( - \Dp^{ij} -\Dr^{ij}- \Ds^{ij} + \Rs^{ij} \right), 
\end{align}
where $\tilde{c}_{ij}:=c_{ij}\,(m_i\;\!m_j)^{-3}\,(m_i+m_j)^3$, and 
\begin{align*}
&\Dp^{ij} := \frac{m_i\;\!m_j}{(m_i+m_j)^2}\int_{\R^{2d}} \alpha_{ij}\;\! F_{ij}\left|\nabla_z\nabla_\sigma\left(\log F_{ij}\right)\right|^2 r^{d-1}\dif z\dif r\dif\sigma, \\
&\Dr^{ij} := \int_{\R^{2d}} F_{ij}\left|\partial_r\!\left[\sqrt{\alpha}_{ij}\;\!\nabla_\sigma\left(\log F_{ij}\right)\right]\right|^2 r^{d-1}\dif z\dif r\dif\sigma, \\
&\Ds^{ij} := \int_{\R^d\times\R_+} \alpha_{ij}\;\!r^{d-3}\left( \int_{\Sp^{d-1}} F_{ij}\, \Gamma_2(\log F_{ij}) \dif\sigma\right)\dif z\dif r,\\
&\Rs^{ij} := \frac{\gamma_{ij}^2}{4} \int_{\R^d\times\R_+} \alpha_{ij}\;\!r^{d-3}\left( \int_{\Sp^{d-1}} F_{ij}\, \Gamma(\log F_{ij}) \dif\sigma\right)\dif z\dif r. 
\end{align*}
\end{proposition}

\begin{proof}
Given any differentiable function $\Phi:\R^{2d}\to\R$, under the change of variables \eqref{change}, the weighted squared gradient decomposes as
\begin{align*}
\frac{1}{m_i}|\nabla_v\Phi|^2 +\frac{1}{m_j}|\nabla_{v_*}\Phi|^2 
=\frac{1}{m_i+m_j}\,|\nabla_z\Phi|^2 + \frac{m_i+m_j}{m_i\;\!m_j}\,|\nabla_{r\sigma}\Phi|^2. 
\end{align*}
Here we know further from $\sigma\cdot\nabla_\sigma\Phi=0$ that 
\begin{align*}
|\nabla_{r\sigma}\Phi|^2=|\partial_r\Phi|^2 +r^{-2}|\nabla_\sigma\Phi|^2. 
\end{align*}
Applying these two formulas with $\Phi=\sqrt{\alpha}_{ij}\;\!\LL\Psi_{ij}=\sqrt{\alpha}_{ij}\;\!\LL\log F_{ij}$ to Proposition~\ref{diss-multiL} yields a decomposition of the Fisher information dissipation into parallel, radial, and spherical parts, 
\begin{align*}
\frac{\dif}{\dif t}\,\II(\ff) = \sum_{i,j=1}^S -D_{{\rm parallel},\;\!c}^{ij} -D_{{\rm radial},\;\!c}^{ij} - D_{{\rm spherical},\;\!c}^{ij} + R_{{\rm\;\!spherical},\;\!c}, 
\end{align*}
where we defined
\begin{align*}
&D_{{\rm parallel},\;\!c}^{ij} := \sum_{k=1}^M\int_{\R^{2d}} \frac{c_{ij}\,\alpha_{ij}}{m_i+m_j}\, F_{ij}\, \big|\nabla_z\LL\log F_{ij}\big|^2 \dif v\dif v_*, \\
&D_{{\rm radial},\;\!c}^{ij} := \sum_{k=1}^M \int_{\R^{2d}} \frac{c_{ij}\left(m_i+m_j\right)\,}{m_i\;\!m_j}\,F_{ij}\, \big|\partial_r\big(\sqrt{\alpha}_{ij}\;\!\LL\log F_{ij}\big)\big|^2 \dif v\dif v_*, \\
&D_{{\rm spherical},\;\!c}^{ij}:= \sum_{k=1}^M \int_{\R^{2d}} \frac{c_{ij}\left(m_i+m_j\right)\alpha_{ij}}{m_i\;\!m_j\;\!r^2}\, F_{ij}\,\big|\nabla_\sigma\LL\log F_{ij}\big|^2 \dif v\dif v_*,\\
&R_{{\rm\;\!spherical},\;\!c}:= \sum_{k=1}^M \int_{\R^{2d}} \frac{c_{ij}\left(m_i+m_j\right)}{m_i\;\!m_j}\frac{(\alpha_{ij}')^2}{4\alpha_{ij}} \,F_{ij}\,\big|\LL\log F_{ij}\big|^2 \dif v\dif v_*.   
\end{align*}
Taking the definition of $\alpha_{ij}$ into account, we see that 
\begin{align*}
\frac{(\alpha_{ij}')^2}{\alpha_{ij}} = \frac{\gamma_{ij}^2\;\!\alpha_{ij}}{r^2}. 
\end{align*}
By the change of variables \eqref{change} and the identity \eqref{LLr}, we arrive at  
\begin{align*}
&D_{{\rm parallel},\;\!c}^{ij} = \tilde{c}_{ij}\, \frac{m_i\;\!m_j}{(m_i+m_j)^2}\int_{\R^{2d}} \alpha_{ij}\;\! F_{ij}\left|\nabla_z\nabla_\sigma\left(\log F_{ij}\right)\right|^2 r^{d-1}\dif z\dif r\dif\sigma, \\
&D_{{\rm radial},\;\!c}^{ij} = \tilde{c}_{ij} \int_{\R^{2d}} F_{ij}\left|\partial_r\!\left[\sqrt{\alpha}_{ij}\;\!\nabla_\sigma\left(\log F_{ij}\right)\right]\right|^2 r^{d-1}\dif z\dif r\dif\sigma, \\
&D_{{\rm spherical},\;\!c}^{ij} = \tilde{c}_{ij}\;\! \sum_{k=1}^M \int_{\R^{2d}} \alpha_{ij}\;\! F_{ij} \left|\nabla_\sigma\left[b_k\cdot\nabla_\sigma\left(\log F_{ij}\right)\right]\right|^2 r^{d-3} \dif z\dif r\dif\sigma,\\
&R_{{\rm\;\!spherical},\;\!c} = \tilde{c}_{ij}\;\!\frac{\gamma_{ij}^2}{4} \int_{\R^{2d}} \alpha_{ij}\;\!F_{ij}\left|\nabla_\sigma(\log F_{ij})\right|^2 r^{d-3}\dif z\dif r\dif\sigma, 
\end{align*}
where we set $\tilde{c}_{ij}:=c_{ij}\,(m_i\;\!m_j)^{-3}\,(m_i+m_j)^3$. According to the Bakry-\'Emery formalism (see \eqref{BE1} and \eqref{BE2}), one finds  
\begin{align*}
D_{{\rm spherical},\;\!c}^{ij} &=  \tilde{c}_{ij}\int_{\R^d\times\R_+} \alpha_{ij}\;\!r^{d-3}\left( \int_{\Sp^{d-1}} F_{ij}\, \Gamma_2(\log F_{ij}) \dif\sigma\right)\dif z\dif r,\\
R_{{\rm\;\!spherical},\;\!c}^{ij} &=  \tilde{c}_{ij}\;\!\frac{\gamma_{ij}^2}{4} \int_{\R^d\times\R_+} \alpha_{ij}\;\!r^{d-3}\left( \int_{\Sp^{d-1}} F_{ij}\, \Gamma(\log F_{ij}) \dif\sigma\right)\dif z\dif r. 
\end{align*}
This concludes the reformulation, with the constant $ \tilde{c}_{ij}$ extracted. 
\end{proof}

\subsection{Functional inequalities on the sphere}
In the light of Proposition~\ref{Fisher-decom}, the nonnegativity of the Fisher information dissipation can be reduced to a spherical geometric criterion. 

The following lemma was proved in \cite{Guillen-Silvestre, Ji}. In the absence of the antipodal symmetry condition, the inequality below is sharp and classical; see \cite{BGL}.  

\begin{lemma}
For any $C^2$-function $\varphi:\Sp^{d-1} \to\R_+$, we have 
\begin{align}\label{sym}
\int_{\Sp^{d-1}} \varphi\,\Gamma_2(\log\varphi)
\ge (d-1)\int_{\Sp^{d-1}}\varphi\,\Gamma(\log\varphi). 
\end{align}
If additionally $\varphi(\sigma) =\varphi(-\sigma)$, then we have 
\begin{align}\label{asym}
\int_{\Sp^{d-1}} \varphi\,\Gamma_2(\log\varphi)
\ge \left(d+3-\frac{1}{d-1}\right)\int_{\Sp^{d-1}}\varphi\,\Gamma(\log\varphi). 
\end{align}
\end{lemma}

We are now in a position to conclude the main result. 

\begin{proof}[Proof of Theorem~\ref{I}]
Fix $(z,r)\in\R^d\times\R_+$ and $i\in\{1,\ldots,S\}$. Recalling that $F_{ii}(v,v_*)=F_{ii}(v_*,v)$, the function $F_{ii}$ is antipodally symmetric in $\sigma$. Therefore, the inequality \eqref{sym} applies to the corresponding term in the decomposition of Proposition~\ref{Fisher-decom} with the choice $\varphi=F_{ii}$.

By contrast, for $i\neq j$ the function $F_{ij}$ need not enjoy this symmetry; equivalently, $F_{ij}(v,v_*)$ is not necessarily invariant under the swap $(v,v_*)\mapsto(v_*,v)$, which corresponds to $\sigma\mapsto-\sigma$. Hence \eqref{sym} cannot be invoked for a general cross-species term. In this case we instead apply the asymmetric inequality \eqref{asym} to the decomposition in Proposition~\ref{Fisher-decom}, for the same fixed $(z,r)$ and any pair $i,j\in\{1,\ldots,S\}$, with the general choice $\varphi=F_{ij}$. This yields the admissible range
\begin{align*}
|\gamma_{ij}|\le \sqrt{4(d-1)}.
\end{align*}
\end{proof}

\end{document}